\documentclass[reqno,twoside,11pt]{amsart}

\usepackage{latexsym,amssymb,amsmath,amsfonts,epsfig,textcomp,mathdots,times,mathptm}

\def\demo{\noindent{\bf Proof. }}
\def\sqr#1#2{{\vcenter{\hrule height.#2pt
            \hbox{\vrule width.#2pt height#1pt \kern#1pt
                    \vrule width.#2pt}
            \hrule height.#2pt}}}
\def\square{\mathchoice\sqr64\sqr64\sqr{4}3\sqr{3}3}
\def\QED{\hfill$\square$}

\def\tratto{\mbox{\rule{2mm}{.2mm}$\;\!$}}
\def\m{{\mathfrak m}}
\def\I{\overline{I}}
\newtheorem{theorem}{Theorem}[section]

\newtheorem{proposition}[theorem]{Proposition}
\newtheorem{discussion}[theorem]{Discussion}

\newtheorem{remark}[theorem]{Remark}
\newtheorem{example}[theorem]{Example}

\begin{document}

\baselineskip=16pt

\title[Normal Hilbert coefficients]{\bf Bounds on the normal Hilbert coefficients } 

\author[A. Corso, C. Polini and M.E. Rossi]
{Alberto Corso \and Claudia Polini \and Maria Evelina Rossi}

\address{Department of Mathematics, University of Kentucky,
Lexington, Kentucky 40506 - USA} \email{alberto.corso@uky.edu}

\address{Department of Mathematics, University of Notre Dame,
Notre Dame, Indiana 46556 - USA} \email{cpolini@nd.edu}

\address{Dipartimento di Matematica, Universit\`a di Genova, Via Dodecaneso 35,
16132 Genova - Italy} \email{rossim@dima.unige.it}



\thanks{The second author was partially supported by NSF grant DMS-1202685 and NSA grant H98230-12-1-0242.}

\begin{abstract}
In this paper we consider extremal and almost extremal bounds on the normal Hilbert coefficients of 
${\mathfrak m}$-primary ideals of an analytically unramified Cohen-Macaulay ring $R$ of dimension 
$d>0$ and infinite residue field. In these circumstances we show that the associated graded ring of the 
normal filtration of the ideal is either Cohen-Macaulay or almost Cohen-Macaulay.
\end{abstract}

\vspace{-0.1in}

\maketitle

\vspace{-0.2in}

\section{Introduction}

The examination of the asymptotic properties of ${\mathfrak m}$-primary ideals of a
Cohen-Macaulay local ring $(R, {\mathfrak m})$ of dimension $d$ and infinite residue field 
has evolved into a challenging area of research, touching most aspects
of commutative algebra, including its interaction with
computational algebra and algebraic geometry. It takes expression
in two graded algebras attached to $I$:  the {\it Rees algebra} 
${\mathcal R}={\mathcal R}(I)$ and  the {\it associated graded
ring} ${\mathcal G}={\mathcal G}(I)$; namely,
\[
{\mathcal R}= \bigoplus_{k=0}^{\infty} I^kt^k \subset R[t], \qquad \text{and} \qquad {\mathcal G} =
{\mathcal R}/I\,{\mathcal R} = \bigoplus_{k=0}^{\infty} I^k/I^{k+1},
\]
where $R[t]$ is the polynomial ring in the variable $t$ over $R$.
These two graded objects are collectively referred to as {\it blowup algebras} of $I$
as they play a crucial role in the process of blowing up the
variety ${\rm Spec}(R)$ along the subvariety $V(I)$.

\medskip

A successful approach in the study of the ring-theoretic
properties of the blowup algebras uses a minimal reduction of
the ideal. This notion was first introduced and exploited by Northcott and Rees more than half a
century ago for its effectiveness in studying multiplicities in
local rings \cite{NR}: an ideal $J$ is a reduction of $I$ if the
inclusion of Rees algebras ${\mathcal R}(J) \hookrightarrow {\mathcal R}(I)$ is module
finite. 
Since $I$ is also an ${\mathfrak m}$-primary ideal, another pathway to
studying blowup algebras -- and more precisely ${\mathcal G}$
-- is to make use of information encoded in the Hilbert-Samuel
function of $I$, that is the function that measures the growth of
the length of $R/I^n$, denoted $\lambda(R/I^n)$, for all $n \geq
1$. For $n \gg 0$, it is known that $\lambda(R/I^n)$ is a
polynomial in $n$ of degree $d$, whose normalized coefficients $e_i=e_i(I)$ are called the Hilbert
coefficients of $I$. The general philosophy, pioneered by Sally in a sequence of remarkable 
papers (see \cite{S,S1,S2,S3,S4}), is that an `extremal' behavior of bounds
involving the $e_i$'s yields good depth properties of the
associated graded ring of $I$. 
These results are somewhat unexpected since the Hilbert coefficients
encode asymptotic information on the Hilbert-Samuel function of $I$. 
The literature is very rich of results, especially relating $e_0$ through $e_3$ to other data of our ideal. We refer to 
the monograph by Rossi and Valla \cite{RV} for a collective overview.

\medskip

Over the years, various filtrations other than the $I$-adic one have proved to be of far reaching applications
in commutative algebra: the integral closure filtration, the Ratliff-Rush filtration, 
the tight closure filtration, and the symbolic power filtration, just to name a few. 
It is important to observe that the theory and results that are valid in the case of the $I$-adic filtration cannot 
be trivially extended to these other types of filtrations. These in fact are not, in general, good or stable filtrations. 
In other words, the Rees algebra associated to these filtrations may not be generated in degree one or may even fail to be Noetherian.
 
\medskip

The focus of our paper is on the significance of the asymptotic properties encoded in the Hilbert function of the integral closure 
filtration of an ${\mathfrak m}$-primary ideal $I$, namely $\{ \overline{I^n} \}$. 
In addition to the local Cohen-Macaulay property of our ambient ring $R$ we also require it to be analytically unramified, that is, 
its ${\mathfrak m}$-adic completion $\widehat{R}$ is reduced. This latter assumption guarantees that the normalization 
$\overline{{\mathcal R}}$ of the Rees algebra ${\mathcal R}$ of $I$ in $R[t]$ is Noetherian (see \cite{Rees}). Hence we have that 
$\lambda(R/\overline{I^{n+1}})$ is a polynomial in $n$ of degree $d$ for $n \gg 0$
\[
\lambda(R/\overline{I^{n+1}}) = \overline{e}_0{ n+d \choose d } - \overline{e}_1{ n+d-1 \choose d-1} + \cdots + (-1)^d \overline{e}_d.
\]
The above polynomial is referred to as the normal Hilbert polynomial of $I$ and the $\overline{e}_i=\overline{e}_i(I)$'s are the normal Hilbert coefficients.
We note that $e_0=\overline{e}_0$ is the multiplicity of $I$.
As in the case of the $I$-adic filtration, there has been quite some interest in relating bounds among the normalized Hilbert coefficients
and the depth of the associated graded ring of the normal filtration of $I$, denoted $\overline{\mathcal G}$.
The forerunners of results along this line of investigation are Huneke (see \cite{H}) and Itoh (see \cite{It,I}).  We refer again to \cite{RV} for a detailed 
account of the results, which typically deal with the first few normal Hilbert coefficients, as in the $I$-adic case.

\medskip

We now describe the contents of our paper. In Section 2 we first introduce a version of the Sally module for the normal filtration. 
We prove in Proposition~\ref{dimension_sally} analogous homological properties to the ones of the classical Sally module 
(see \cite{V}). This allows us to recover, in Proposition~\ref{prop2.3}, the bound $ \overline{e}_1 \geq e_0 - \lambda(R/\overline{I})$, established by 
Huneke \cite[theorem 4.5]{H} and Itoh \cite[Corollary]{I}. That the  equality in the bound is equivalent to the Cohen-Macaulayness of 
$\overline{{\mathcal G}}$ translates in our setting to the vanishing of the variant of the Sally module. We then show in 
Theorem~\ref{theorem2.6} the main result of this section. Namely, we show that if the previous bound is almost extremal, that is 
$\overline{e}_1 \leq e_0 - \lambda(R/\overline{I})+1$, then the depth of $\overline{{\mathcal G}}$ is at least $d-1$.

\medskip

In \cite{I}, Itoh already established lower bounds on $\overline{e}_2$ and $\overline{e}_3$. 
More precisely, he showed that $\overline{e}_2 \geq \overline{e}_1 - \lambda(\overline{I}/J)$ with equality if and only if the normal filtration of $I$ 
has reduction number two (see \cite[Theorem~2(2)]{I}). In particular $\overline{\mathcal G}$ is Cohen-Macaulay. He also showed that 
$\overline{e}_3 \geq 0$ (see \cite[Theorem~3(1)]{I}). When $R$ is Gorenstein and $\overline{I}={\mathfrak m}$ he was able to conclude that 
the vanishing of $\overline{e}_3$ is equivalent to the normal filtration of $I$  having reduction number two (see \cite[Theorem~3(2)]{I}). 
Again this implies that $\overline{\mathcal G}$ is Cohen-Macaulay.

In Section 3 we generalize Itoh's result on the vanishing of $\overline{e}_3$ by considering arbitrary 
Cohen-Macaulay rings of type $t(R)$ together with the assumption $\lambda(\overline{I^2}/J\overline{I}) \geq t(R)-1$ 
(see Theorem~\ref{theorem3.3}).  Obviously the latter is a vacuous assumption  when $R$ is Gorenstein. We also 
extend Itoh's result with no further assumptions to rings of type at most two  (see Theorem~\ref{theorem3.5}). 
We note that a condition on the type of the ring is not unexpected as it is reminiscent of the celebrated result 
of Sally \cite[Theorem~3.1]{S4} on  Cohen-Macaulay rings of type $e+d-2$.

\bigskip

\section{The Sally module and Hilbert coefficients of the normal filtration}

One of the first inequalities involving the Hilbert coefficients of an ${\mathfrak m}$-primary ideal $I$ goes back to 1960, 
when Northcott \cite{No} showed that $e_1-e_0+\lambda(R/I) \geq 0$. Later it was shown by Huneke \cite{H} and Ooishi \cite{Oo}
that equality  is  equivalent to the ideal having reduction number one for any minimal reduction $J$ of $I$, that is $I^2=JI$. 
Hence ${\mathcal G}$ is Cohen-Macaulay. 

An elegant and theoretical explanation of the results by Northcott, Huneke and Ooishi was captured by 
Vasconcelos \cite{V} with the introduction of a new graded object: the so-called Sally module. 
As noted in the monograph \cite{RV}, the Sally module can actually be defined for an 
arbitrary filtration. However, additional properties on the filtration are needed to be able to do the extra mile.
This is the case in this article where we consider the normal filtration of an ideal.

The Sally module of the normal filtration of an ${\mathfrak m}$-primary ideal $I$ with minimal reduction $J$ 
is defined by the short exact sequence of ${\mathcal R}(J)$-modules
\begin{eqnarray} \label{eq_sally}
0 \rightarrow \overline{I} {\mathcal R}(J) \longrightarrow \overline{{\mathcal R}}_{\geq 1}[-1] \longrightarrow
\overline{{\mathcal S}} \rightarrow 0.
\end{eqnarray}
More explicitly, one has
\[
\overline{{\mathcal S}} = \bigoplus_{n \geq 1} \overline{I^{n+1}}/J^{n}\overline{I}.
\]

In Proposition~\ref{dimension_sally} we establish  a key homological property of $\overline{{\mathcal S}}$ that will 
be used in Theorem~\ref{theorem2.6}, the main theorem of this section. This is the same property  of the classical Sally module. 
However it does not follow directly from the original result because the normal filtration is not multiplicative.

\begin{proposition}\label{dimension_sally}
Let $(R, {\mathfrak m})$ be an analytically unramified, Cohen-Macaulay local ring of dimension $d>0$ and infinite residue field. 
Let $I$ be an ${\mathfrak m}$-primary ideal, $J$ a minimal reduction of $I$, and
$\overline{\mathcal S}$ the Sally module of the normal filtration of $I$ 
with respect to $J$, as defined in \mbox{\rm (\ref{eq_sally})}. Then
$\overline{{\mathcal S}}$ is a nonzero module if and only if ${\rm Ass}_{{\mathcal R}(J)}(\overline{{\mathcal S}})=
\{ \, {\mathfrak m} {\mathcal R}(J) \, \}$. In particular, $\overline{\mathcal S}$ has dimension $d$.
\end{proposition}
\begin{proof}
Let us assume that $\overline{\mathcal S} \not=0$, as the other implication is trivial. 
As ${\rm Ass}_{{\mathcal R}(J)}(\overline{\mathcal S}) \not= \emptyset$, let $P \in 
{\rm Ass}_{{\mathcal R}(J)}(\overline{\mathcal S})$ and write $P \cap R ={\mathfrak p}$. If ${\mathfrak p} \not=
{\mathfrak m}$ it follows that $\overline{\mathcal S}_{\mathfrak p}=0$,  from (\ref{eq_sally}) and the fact that 
$(\overline{I}{\mathcal R}(J))_{\mathfrak p} = (\overline{\mathcal R}_{\geq 1}[-1])_{\mathfrak p}$. Hence ${\mathfrak p}={\mathfrak m}$ 
and $P \supseteq {\mathfrak m}{\mathcal R}(J)$. 

We claim that $P={\mathfrak m}{\mathcal R}(J)$. Notice that 
${\mathfrak m}{\mathcal R}(J)$ is a prime of height $1$, thus if $P \supsetneq {\mathfrak m}{\mathcal R}(J)$ we have that 
$P$ is a prime ideal of height at least $2$. A depth computation is the short exact sequence (\ref{eq_sally}) 
shows that ${\rm depth} \ \overline{\mathcal S}_P \geq 1$, which contradicts the fact that $P$ is an associated prime of $\overline{\mathcal S}$.

The asserted depth estimate follows since $\overline{I} {\mathcal R}(J)$ is maximal Cohen-Macaulay and 
$\overline{\mathcal R}_{\geq 1}$ has property $S_2$ of Serre. The first assertion is a consequence 
of the short exact sequence
\[
0 \rightarrow \overline{I} {\mathcal R}(J) \longrightarrow {\mathcal R}(J) \longrightarrow
{\mathcal R}(J)/\overline{I} {\mathcal R}(J)  \rightarrow 0
\]
and the fact that the module ${\mathcal R}(J)/\overline{I} {\mathcal R}(J)$ is isomorphic to $R/\overline{I}[T_1, \ldots, T_d]$.
Now, the short exact sequence 
\[
0 \rightarrow \overline{\mathcal R}_{\geq 1} \longrightarrow \overline{\mathcal R} \longrightarrow
R  \rightarrow 0,
\]
establishes instead the second assertion, since $\overline{\mathcal R}$ has property $S_2$ of Serre and 
$R$ is Cohen-Macaulay  of dimension $d>0$. 
\end{proof}

Our next goal is to determine the relationship between the coefficients of the Hilbert polynomial of the 
Sally module $\overline{\mathcal S}$ and the ones of the Hilbert polynomial of the associated graded 
ring $\overline{\mathcal G}$ of the normal filtration of $I$. The construction we describe allows us, in particular, to 
give a concrete characterization of the {\it sectional normal genus} $g_s=\overline{e}_1-e_0+\lambda(R/\overline{I})$, 
defined by Itoh in \cite{I},  as the multiplicity
of the Sally module $\overline{\mathcal S}$. We then use this characterization to give an estimate of the depth
of the associated graded ring $\overline{\mathcal G}$.

\begin{discussion}{\rm
Following the construction of \cite[Proposition 6.1]{RV}, there exists a graded module $N$ which fits
in the short exact sequences
\begin{equation}\label{eq1}
0 \rightarrow {\mathcal G}({\mathbb E}) \longrightarrow N
\longrightarrow \overline{{\mathcal S}} [-1] \rightarrow 0
\end{equation}
\begin{equation}\label{eq2}
0 \rightarrow \overline{{\mathcal S}} \longrightarrow N
\longrightarrow \overline{\mathcal G} \rightarrow 0.
\end{equation}
More precisely, ${\mathbb E}$ denotes the $J$-good filtration 
$\{E_0=R, E_n=J^{n-1}\overline{I} \text{ for all } n \geq 1 \}$ induced by the $R$-ideal $\overline{I}$, 
${\mathcal G}({\mathbb E}) $ denotes the associated graded ring of ${\mathbb E}$, $\overline{\mathcal G}$ 
denotes the associated graded ring of the normal filtration of $I$, and, finally, $N$ denotes the graded 
module $\displaystyle\bigoplus_{n\geq 0} \overline{I^n}/J^n\overline{I}$.

From the results of \cite[Sections 6.1 and 6.2]{RV} modified to suit our situation, we have that the Hilbert 
series of the graded modules appearing in (\ref{eq1}) and (\ref{eq2}) are related by the equation
\begin{eqnarray} \label{eq4}
(1-z) HS_{\overline{\mathcal S}}(z) = HS_{{\mathcal G}({\mathbb E})}(z)-HS_{\overline{\mathcal G}}(z).
\end{eqnarray}
Furthermore, ${\mathcal G}({\mathbb E})$ is Cohen-Macaulay with minimal multiplicity since $E_{n+1}=JE_n$ for all $n\geq 1$. Hence
its Hilbert series is given by \cite[(6.1)]{RV}
\[
HS_{{\mathcal G}({\mathbb E})}(z) = \frac{\lambda(R/\overline{I})+(e_0-\lambda(R/\overline{I}))z}{(1-z)^d}.
\] }
\end{discussion}

By combining all this information we obtain the following result.

\begin{proposition}\label{prop2.3}
Let $(R, {\mathfrak m})$ be an analytically unramified, Cohen-Macaulay local ring of dimension $d>0$ 
and infinite residue field. Let $I$ be an ${\mathfrak m}$-primary ideal, 
$J$ a minimal reduction of $I$, and $\overline{\mathcal S}$ the Sally module of the normal filtration of $I$ 
with respect to $J$. Let $\overline{\mathcal G}$ denote the associated graded ring of the normal filtration of $I$.
Let $\overline{s}_i$ and $\overline{e}_i$ denote the normalized Hilbert coefficients of $\overline{\mathcal S}$ and 
$\overline{\mathcal G}$, respectively. The following properties hold:
\begin{itemize}
\item[$(a)$]
if \/ $\overline{{\mathcal S}}=0$ then $\overline{{\mathcal G}}$ is Cohen-Macaulay\/$;$

\item[$(b)$]
if \/ $\overline{{\mathcal S}} \not= 0$ then ${\rm depth} \ {\overline{\mathcal G}} \geq {\rm depth} \ {\overline{\mathcal S}} -1$\/$;$

\item[$(c)$]
$\overline{s}_0 = \overline{e}_1-e_0 + \lambda(R/\overline{I})$ and 
$\overline{s}_i = \overline{e}_{i+1}$ for all $1 \leq i \leq d-1$.
\end{itemize}
\end{proposition}
\begin{proof}
$(a)$  follows from (\ref{eq1}) and (\ref{eq2}) because $\overline{\mathcal G} \cong N \cong {\mathcal G}({\mathbb E})$, 
whenever $\overline{\mathcal S}=0$, and ${\mathcal G}({\mathbb E})$ is Cohen-Macaulay. 
$(b)$ follows from depth chase in (\ref{eq1}) and (\ref{eq2}).
Finally, $(c)$ follows from (\ref{eq4}) and the fact, shown in Proposition~\ref{dimension_sally}, 
that if $\overline{\mathcal S} \not=0$ then its dimension is $d$.
\end{proof}

\begin{remark}{\rm
Proposition~\ref{prop2.3}$(c)$ provides a simple proof of the bounds
\[
\overline{e}_1\geq e_0-\lambda(R/\overline{I})=\lambda(\overline{I}/J)\geq 0.
\] 
Moreover, Proposition~\ref{prop2.3}$(a)$ is equivalent to the equality $\overline{e}_1=e_0-\lambda(R/\overline{I})$; 
it is also equivalent to the isomorphism $\overline{\mathcal G} \cong {\mathcal G}({\mathbb E})$; it is also equivalent to 
the fact that the reduction number of the integral closure filtration is at most $1$ (see \cite[Theorem~4.5]{H} and \cite[Corollary 6]{I}).}
\end{remark}

In the main theorem of this section we study the relation between an upper bound on $\overline{e}_1$ and 
the depth of $\overline{\mathcal G}$. That is, we investigate the depth property of 
$\overline{\mathcal G}$ whenever $\overline{e}_1 \leq e_0-\lambda(R/\overline{I})+1$. This is equivalent to assuming that the multiplicity 
of the Sally module $\overline{\mathcal S}$ is at most one. (See also \cite[Proposition~3.5]{V} and \cite[Corollary~3.7]{V} in the $I$-adic case.)

\begin{theorem}\label{theorem2.6}
Under the same assumptions as in {\rm Proposition~\ref{prop2.3}}, 
if \/ $\overline{e}_1  \leq e_0-\lambda(R/\overline{I})+1$ then ${\rm depth} \, \overline{\mathcal G} \geq d-1$.
\end{theorem}
\begin{proof}
%
By Proposition~\ref{prop2.3}$(c)$, our assumption on $\overline{e}_1$ implies that either $\overline{e}_1= e_0-\lambda(R/\overline{I})$ 
or $\overline{e}_1= e_0-\lambda(R/\overline{I})+1$. In the first case we have that the Sally module $\overline{\mathcal S}$ is zero and hence 
$\overline{\mathcal G}$ is Cohen-Macaulay by Proposition~\ref{prop2.3}$(a)$. Thus we are left to consider the second case.
From $\overline{e}_1 =e_0-\lambda(R/\overline{I})+1$ we obtain that  $\overline{\mathcal S}$ is a nonzero module of multiplicity one.
By Proposition~\ref{dimension_sally} we have that ${\rm Ass}_{{\mathcal R}(J)}(\overline{\mathcal S}) = \{ \, {\mathfrak m}{\mathcal R}(J) \, \}$. 
Thus $\overline{\mathcal S}$ is a torsion free $B$-module of rank one, where $B={\mathcal R}(J)/{\mathfrak m}{\mathcal R}(J)$ is a polynomial ring in 
$d$ variables over the residue field. 
We claim that $\overline{\mathcal S}$  is a reflexive $B$-module, hence it is free since $B$ is a UFD. 
In particular, ${\rm depth}\, \overline{\mathcal S}=d$. Hence, by 
Proposition~\ref{prop2.3}$(b)$, we conclude that ${\rm depth}\, \overline{\mathcal G} \geq d-1$.
Our claim is equivalent to the fact that $\overline{\mathcal S}$ has property $S_2$ of Serre as a $B$-module. 
As ${\rm Ass}_B(\overline{{\mathcal S}})=\{ 0 \}$ it suffices to show ${\rm depth} \, \overline{{\mathcal S}}_P \geq 2$ for each 
$P \in {\rm Spec}(B)$ with height at least two. Let $Q \in {\rm Spec}({\mathcal R}(J))$ be such that 
$P = Q/{\mathfrak m}{\mathcal R}(J)$. 

As we observed in the proof of Proposition~\ref{dimension_sally}, $\overline{I} {\mathcal R}(J)$ is maximal Cohen-Macaulay and 
$\overline{\mathcal R}_{\geq 1}$ has property $S_2$ of Serre, hence depth chasing in (\ref{eq_sally}) yields 
${\rm depth} \ \overline{{\mathcal S}}_Q \geq 2$. Thus ${\rm depth} \ \overline{{\mathcal S}}_P = {\rm depth} \ 
\overline{{\mathcal S}}_Q \geq 2$.
\end{proof}

%

\begin{remark}{\rm
Theorem~\ref{theorem2.6} would be a consequence of  \cite[Proposition 4.9]{MMV}. However the proof of  
\cite[Proposition 4.9]{MMV}  is not correct as it relies on \cite[Theorem 3.26]{Marley}  which is incorrectly stated. 
Indeed in dimension $d>1$ it is not clear that $\overline{I^2}=J\overline{I}$ implies that the normal filtration 
has reduction number at most one, that is $\overline{I^{n+1}}=J\overline{I^n}$ for all $n\geq 1$. }
\end{remark}

\section{On the vanishing of $\overline{e}_3$}

As we mentioned in the introduction Itoh showed that the vanishing of $\overline{e}_3$ is equivalent to the normal filtration of $I$  
having reduction number two, provided $R$ is Gorenstein and $\overline{I}={\mathfrak m}$  (see \cite[Theorem~3(2)]{I}). 
Again this implies that $\overline{\mathcal G}$ is Cohen-Macaulay.
We now generalize Itoh's result on the vanishing of $\overline{e}_3$ by considering arbitrary Cohen-Macaulay and imposing a 
condition on the type $t(R)$ of the ring.

\begin{proposition}\label{proposition3.1}
Let $(R, {\mathfrak m})$ be an analytically unramified, Cohen-Macaulay local ring of dimension $d \geq 3$, type $t(R)$ and infinite residue field. 
Let $I$ be a $R$-ideal with $\overline{I}={\mathfrak m}$ and 
$J$ a minimal reduction of $I$. Assume that $\overline{e}_3=0$. Then
 \[
 \lambda(\overline{I^{n+1}}/J^n\I) \le  t(R){ {n+d-2}\choose{d-1}}
 \]
for all $n\ge 1$. In particular, $\lambda(\overline{I^{2}}/J\I) \le  t (R)$.
\end{proposition}
\begin{proof}
By \cite[Theorem 3(1)]{I} the assumption $\overline{e}_3=0$ yields $\overline{I^{n+2}}\subset J^n$ for all $n \ge 0$. By assumption we have that $\I=\m$, hence
\[
{\m} \overline{I^{n+1}}=\overline{I} \, \overline{I^{n+1}}\subset \overline{I^{n+2}} \subset J^n.
\]
This implies that $\overline{I^{n+1}} \subset J^n : \m$, thus
\begin{eqnarray*}
\lambda(\overline{I^{n+1}}/J^n\I) &=&\lambda(\overline{I^{n+1}}/(J^n \cap \overline{I^{n+1}}))  = \lambda(\overline{I^{n+1}}+J^n/J^n)  \\
&\le& \lambda(J^n \colon \m/J^n) = t(R){ {n+d-2}\choose{d-1}}.
\end{eqnarray*}
We observe that in the first equality we used \cite[Theorem 1]{It} or \cite[Theorem 4.7 and Appendix]{H} whereas
the last equality holds because  $\lambda(J^n \colon \m/J^n)$ is the dimension of the socle of the ring $R/J^n$, which can be computed 
from the Eagon-Northcott resolution of $R/J^n$.

The second inequality asserted in the theorem follows by setting $n=1$ in the general inequality.
\end{proof}

In the next result we present both a lower and upper bound on $\overline{e}_1$  under the running assumptions of this 
section, namely, $\overline{I}={\mathfrak m}$ and $\overline{e}_3=0$. The lower bound appears already in \cite[Theorem~2(1)]{I}. 
We also establish a condition that assures that the upper bound is tight.

\begin{proposition}\label{proposition3.2}
Let $(R, {\mathfrak m})$ be an analytically unramified, Cohen-Macaulay local ring of dimension $d \geq 3$, type $t(R)$ 
and infinite residue field. 
Let $I$ be a $R$-ideal with $\overline{I}={\mathfrak m}$ and $J$ a minimal reduction of $I$.
Assume that  $\overline{e}_3=0$. Then
\[
e_0-1+\lambda(\overline{I^{2}}/J\I) \leq \overline{e}_1 \leq e_0-1+ t(R).
\]
Moreover, if $t(R)\not= \lambda(\overline{I^{2}}/J\I)$ then $ \overline{e}_1 < e_0-1+ t(R)$.
\end{proposition}
\begin{proof}
Notice that $e_0-1=\lambda(\overline{I}/J)$, since $\overline{I}={\mathfrak m}$.
By \cite[Proposition 10]{It} we have for any $n \geq 0$
\begin{equation} \label{equat1}
\lambda(R/\overline{I^{n+1}}) \leq e_0{n+d \choose d} - [\lambda(\overline{I}/J)+\lambda(\overline{I^2}/J\overline{I})] {n+d-1 \choose d-1} +\lambda(\overline{I^2}/J\overline{I}){n+d-2 \choose d-2}.  
\end{equation}
Now
\begin{eqnarray*}
\lambda(R/\overline{I^{n+1}})
& =  & \lambda(R/J^{n+1})-\lambda(J^n\overline{I}/J^{n+1}) - \lambda(\overline{I^{n+1}}/J^n\overline{I}) \\
& =  & \lambda(R/J^{n+1}) - \lambda(J^n/J^{n+1}) + \lambda(J^n/J^n\overline{I}) - \lambda(\overline{I^{n+1}}/J^n\overline{I}) \\
& =  & e_0 {n+d \choose d} -e_0 {n+d-1 \choose d-1} + \lambda(R/\overline{I}) {n+d-1 \choose d-1} - \lambda(\overline{I^{n+1}}/J^n \overline{I})
\end{eqnarray*}
where $\lambda(J^n/J^n\overline{I}) = \lambda(R/\overline{I}) \displaystyle {n+d-1 \choose d-1}$ follows since
\begin{align*}
J^n/J^n \overline{I} & \cong  J^n/J^{n+1} \otimes R/\overline{I}  \cong  [{\rm gr}_J(R)]_n \otimes R/\overline{I} \cong   [{\rm gr}_J(R) \otimes R/\overline{I}]_n \\
& \cong [R/J[T_1, \ldots, T_d] \otimes R/\overline{I}]_n \cong  [R/\overline{I}[T_1, \ldots, T_d]]_n.
\end{align*}
By Proposition~\ref{proposition3.1} it follows that for any $n \geq 0$ we have that
\begin{eqnarray} \label{equat2}
\begin{array}{rcl}
\lambda(R/\overline{I^{n+1}}) & \geq & e_0 \displaystyle {n+d \choose d} -e_0 {n+d-1 \choose d-1} + \lambda(R/\overline{I}) {n+d-1\choose d-1} - t(R) {n+d-2 \choose d-1} \\
& = & e_0 \displaystyle{n+d \choose d} - [\lambda(\overline{I}/J)+t(R)] {n+d-1 \choose d-1} + t(R) {n+d-2 \choose d-2}.
\end{array}
\end{eqnarray}
Note that $\displaystyle{n+d-2 \choose d-1} = {n+d-1 \choose d-1} - {n+d-2 \choose d-2}$ and recall that for all $n \gg 0$ we have
\begin{eqnarray} \label{equat3}
\lambda(R/\overline{I^{n+1}})  = e_0 {n+d \choose d} - \overline{e}_1{n+d-1 \choose d-1}+\overline{e}_2{n+d-2 \choose d-2} + {\rm \ lower \ terms}. 
\end{eqnarray}
Comparing (\ref{equat1}), (\ref{equat2}) and (\ref{equat3}) we obtain immediately that
\[
\lambda(\overline{I}/J)+\lambda(\overline{I^2}/J\overline{I}) \leq \overline{e}_1 \leq \lambda(\overline{I}/J) + t(R).
\]
Now assume that $t(R) \not= \lambda(\overline{I^2}/J\overline{I})$. If in (\ref{equat2}) the inequality is strict for at least one $n \gg 0$,
then comparing (\ref{equat2}) and (\ref{equat3}) we obtain the desired conclusion, that is
\[
\overline{e}_1 < \lambda(\overline{I}/J) + t(R).
\]
Otherwise in (\ref{equat2}) the equality holds for all $n \gg 0$. Again comparing (\ref{equat2}) and (\ref{equat3}) we obtain
\[
\overline{e}_1 = \lambda(\overline{I}/J)+t(R) \qquad \overline{e}_2 = t(R).
\]
Hence $\overline{e}_2=\overline{e}_1 - \lambda(\overline{I}/J)$, which implies
$\overline{I^{n+1}} = J^{n-1} \overline{I^2}$ for all $n \geq 1$ by \cite[Theorem~2(2)]{I}. 
Now by \cite[Proposition 10]{It} we have that equality holds in (\ref{equat1}), hence
\[
t(R) = \overline{e}_2 = \lambda(\overline{I^2}/J \overline{I}),
\]
which is a contradiction.
\end{proof}

In the following theorem we analyze the case when  $\lambda(\overline{I^{2}}/J\I) $ is maximal or almost maximal.
In accordance to the classical philosophy we prove that if the bound is attained then the associated graded ring of the
normal filtration is Cohen-Macaulay and furthermore the normal filtration has reduction number two.

\begin{theorem}\label{theorem3.3}
Let $(R, {\mathfrak m})$ be an analytically unramified, Cohen-Macaulay local ring of dimension $d \geq 3$, type $t(R)$ and infinite residue field. 
Let $I$ be an $R$-ideal with $\overline{I}={\mathfrak m}$ and $J$ a minimal reduction of $I$. Assume that
 $\overline{e}_3=0$ and  $\lambda(\overline{I^{2}}/J\I) \ge  t(R) -1$. Then $\overline{{\mathcal G}}$ is Cohen-Macaulay and 
 $\overline{I^{n+1}}=J^{n-1}\overline{I^2}$ for all $n \geq 1$.
\end{theorem}
\begin{proof}
Using Proposition~\ref{proposition3.1} and our assumption we have that
\[
\lambda(\overline{I^2}/J\overline{I}) \leq t(R) \leq \lambda(\overline{I^2}/J\overline{I}) +1.
\]
If $t(R) = \lambda(\overline{I^2}/J\overline{I})$, then by Proposition~\ref{proposition3.2} we have
$\overline{e}_1 = \lambda(\overline{I}/J) + \lambda(\overline{I^2}/J\overline{I})$.
Hence by \cite[Theorem 2.(1)]{I}, we have $\overline{I^{n+1}}=J^{n-1}\overline{I^2}$
for every $n \geq 1$.

If $t(R) = \lambda(\overline{I^2}/J\overline{I})+1$, then again by Proposition~\ref{proposition3.2} we have
\[
\lambda(\overline{I}/J) + \lambda(\overline{I^2}/J\overline{I}) \leq \overline{e}_1 <  \lambda(\overline{I}/J) + \lambda(\overline{I^2}/J\overline{I}) + 1.
\]
Thus $\overline{e}_1=\lambda(\overline{I}/J)+\lambda(\overline{I^2}/J\overline{I})$ and we conclude as before.

Notice that if $\overline{I^{n+1}}=J^{n-1}\overline{I^2}$, then $\overline{{\mathcal G}}$ is Cohen-Macaulay by the Valabrega-Valla criterion 
(see \cite[Theorem 1.1]{RV}), since $\overline{I^2} \cap J = J \overline{I}$ by \cite[Theorem 1]{It} or \cite[Theorem 4.7 and Appendix]{H}.
\end{proof}

If we strengthen the assumptions in Theorem~\ref{theorem2.6} by adding the vanishing of $\overline{e}_3$ we obtain the 
Cohen-Macaulayness of $\overline{{\mathcal G}}$. Furthermore, the normal filtration of $I$ has reduction number at most $2$.

\begin{proposition}\label{proposition3.4}
Let $(R, {\mathfrak m})$ be an analytically unramified, Cohen-Macaulay local ring of dimension $d \geq 3$ 
and infinite residue field. Let $I$ be an ${\mathfrak m}$-primary ideal and $J$ a minimal reduction of $I$. 
Assume $\overline{e}_1 = e_0 - \lambda(R/\overline{I})+1$ and $\overline{e}_3=0$. Then $\overline{{\mathcal G}}$ is 
Cohen-Macaulay and the normal filtration of $I$ has reduction number at most $2$.
\end{proposition}
\begin{proof}
We have that
\[
\lambda(\overline{I}/J) + 1 = \overline{e}_1 \geq \lambda(\overline{I}/J) + \sum_{n \geq 1} \lambda(\overline{I^{n+1}}/J \cap \overline{I^{n+1}}),
\]
where the equality holds by assumption and the inequality is given by \cite[Corollary 4.8]{HM}. 
Hence $\displaystyle\sum_{n \geq 1} \lambda(\overline{I^{n+1}}/J \cap\overline{ I^{n+1}}) \leq 1$. 
Again by \cite[Corollary 4.8]{HM} equality holds if and only if 
$\overline{{\mathcal G}}$ is Cohen-Macaulay and the reduction number of the normal filtration is at most 2. 

Thus we may assume that $\displaystyle\sum_{n \geq 1} \lambda(\overline{I^{n+1}}/J \cap\overline{ I^{n+1}}) =0$. In particular,
$\overline{I^2} = J \cap \overline{I^2}=J \overline{I}$ by \cite[Theorem 1]{It} or \cite[Theorem 4.7 and Appendix]{H}. 

By Theorem~\ref{theorem2.6} ${\rm depth} \, \overline{{\mathcal G}} \geq d-1$. According to \cite[Proposizione~4.6]{HM} we have
\[
\overline{e}_3 = \sum_{j \geq 2} { j \choose 2} \lambda(\overline{I^{j+1}}/J\overline{I^j}). 
\]
As $\overline{e}_3=0$, we obtain that $\overline{I^{j+1}}=J\overline{I^j}$  for all $j \geq 1$. 
Thus $\overline{{\mathcal G}}$ is Cohen-Macaulay by the Valabrega-Valla criterion (see \cite[Theorem 1.1]{RV}).
\end{proof}

\begin{remark}{\rm 
Notice that the second case in the proof of Proposition~\ref{proposition3.4} cannot happen. In fact one would have that
$\overline{e}_1 > \lambda(\overline{I}/J) + \sum_{n \geq 1} \lambda(\overline{I^{n+1}}/J \cap \overline{I^{n+1}})$ 
and $\overline{{\mathcal G}}$ Cohen-Macaulay, thus contradicting \cite[Corollary 4.8]{HM}. }
\end{remark}

\begin{theorem}\label{theorem3.5}
Let $(R, {\mathfrak m})$ be an analytically unramified, Cohen-Macaulay local ring of dimension $d \geq 3$,
type $t(R)\leq 2$ and infinite residue field. Let $I$ be an $R$-ideal with $\overline{I}={\mathfrak m}$ and $J$ a minimal reduction of $I$.
Assume that  $\overline{e}_3=0$. Then 
\begin{itemize}
\item[$(a)$]
$\overline{{\mathcal G}}$ is Cohen-Macaulay$;$

\item[$(b)$] 
${\mathcal G}({\mathfrak m})$ is Cohen-Macaulay, except in the case \ $t(R)=\lambda(\overline{I^2}/J{\mathfrak m})=\mu({\mathfrak m})-d=2$ \
and 
$\lambda({\mathfrak m}^2/J{\mathfrak m}) =1$. In this latter situation, though, ${\rm depth} \, {\mathcal G}({\mathfrak m}) \geq d-1$. 
\end{itemize}
\end{theorem}
\begin{proof}
Assume $t(R)=1$. From Theorem~\ref{theorem3.3} it follows that $\overline{{\mathcal G}}$ is Cohen-Macaulay (see also \cite[Theorem 3]{I}).
Now $\lambda({\mathfrak m}^2/J{\mathfrak m}) = \lambda({\mathfrak m}^2/J\overline{I}) \leq \lambda(\overline{I^2}/J\overline{I}) \leq 1$
by Proposition~\ref{proposition3.1}. If ${\mathfrak m}^2=J{\mathfrak m}$ then  ${\mathcal G}({\mathfrak m})$
is Cohen-Macaulay (see \cite[Theorems 1 and 2]{S}). If $\lambda({\mathfrak m}^2/J{\mathfrak m}) = \lambda(\overline{I^2}/J{\mathfrak m}) = 1$,
then ${\mathfrak m}^2 = \overline{I^2}$ and hence by Theorem~\ref{theorem3.3} we have that for all $n \geq 1$
\[
{\mathfrak m}^{n+1} = {\overline I^{n+1}} \subseteq \overline{I^{n+1} }=J^{n-1}\overline{I^2} = J^{n-1}{\mathfrak m}^2 \subseteq {\mathfrak m}^{n+1}.
\]
Thus $\overline{{\mathcal G}}={\mathcal G}({\mathfrak m})$ and so  ${\mathcal G}({\mathfrak m})$ is Cohen-Macaulay as well 
(see also \cite[Proposition~3.3 and Theorem~3.4]{S2}).

Assume now $t(R)=2$. By Proposition~\ref{proposition3.1} we have $\lambda(\overline{I^2}/J\overline{I}) \leq t(R)=2$
and, by Theorem~\ref{theorem3.3}, $\overline{{\mathcal G}}$ is Cohen-Macaulay whenever $\overline{I^2} \not= J\overline{I}$.
Assume $\overline{I^2}=J\overline{I}$. By Proposition~\ref{proposition3.2} one has
\[
\lambda(\overline{I}/J) \leq \overline{e}_1 \leq \lambda(\overline{I}/J)+1.
\]
If $\overline{e}_1 = \lambda(\overline{I}/J)$, then $\overline{{\mathcal G}}$ is Cohen-Macaulay by \cite[Theorem 2(1)]{I}. If $\overline{e}_1 =
\lambda(\overline{I}/J)+1$, then the same conclusion follows from Proposition~\ref{proposition3.4}.

Now we are going to study ${\mathcal G}({\mathfrak m})$. As before $\lambda({\mathfrak m}^2/J{\mathfrak m}) \leq \lambda(\overline{I^2}/J\overline{I}) \leq 2$.
If ${\mathfrak m}^2=J{\mathfrak m}$, clearly ${\mathcal G}({\mathfrak m})$ is Cohen-Macaulay (see \cite[Theorems 1 and 2]{S}).

If $\lambda({\mathfrak m}^2/J{\mathfrak m})=2$,
then ${\mathfrak m}^2=\overline{I^2}$.  Since $\overline{I^{n+1}}=J^{n-1}\overline{I^2}$ again by Theorem~\ref{theorem3.3} we have that
${\mathfrak m}^n=\overline{I^n}$ for all $n$, as shown above. In particular $\overline{{\mathcal G}}={\mathcal G}({\mathfrak m})$ and 
${\mathcal G}({\mathfrak m})$ is Cohen-Macaulay as well.

Finally, if $\lambda({\mathfrak m}^2/J{\mathfrak m})=1$ and $\mu({\mathfrak m})-d < t(R)=2$, then by \cite[Theorem~3.1]{S4} 
${\mathcal G}({\mathfrak m})$ is Cohen-Macaulay. Thus we are left to consider the case when $\mu({\mathfrak m})-d=2$ and 
$1=\lambda({\mathfrak m}^2/J{\mathfrak m}) < \lambda(\overline{I^2}/J \overline{I})$.
Otherwise, as before, ${\mathfrak m}^n=\overline{I^n}$ for all $n$, thus  $\overline{{\mathcal G}}={\mathcal G}({\mathfrak m})$ and
${\mathcal G}({\mathfrak m})$ is Cohen-Macaulay as well.
By \cite[Theorem~2.1]{RV0} we can only conclude that ${\rm depth}\, {\mathcal G}({\mathfrak m}) \geq d-1$.
\end{proof}

We conclude our paper by showing that ${\mathcal G}({\mathfrak m})$ fails to be Cohen-Macaulay in the exceptional case described in 
Theorem~\ref{theorem3.5}$(b)$.

\begin{example}{\rm
Consider first the one-dimensional Cohen-Macaulay local ring $S$ of type $2$ and multiplicity $4$ given by the semigroup ring $k[\![t^4, t^5, t^{11}]\!]$,
which can be easily seen to be isomorphic to $k[\![X,Y,Z]\!]/(Z^2-X^3Y^2,Y^3-XZ, X^4-YZ)$.
It was shown by J. Sally in \cite{S} that the associated graded ring ${\mathcal G}({\mathfrak m})$ is not Cohen-Macaulay. Notice now 
that $\overline{{\mathfrak m}^n} = (t^{4n}) k[\![ t ]\!] \cap R$ for all $n\geq 1$ and that the conductor of $R$ is given by $t^8$. Thus, $\overline{{\mathfrak m}^2} = 
{\mathfrak m}^2 + (t^{11})$ whereas $\overline{{\mathfrak m}^n} = {\mathfrak m}^n$ for all $n \geq 3$.
This shows that the associated graded ring $\overline{\mathcal G}$ of the normal filtration of ${\mathfrak m}$ and ${\mathcal G}({\mathfrak m})$ have the same 
Hilbert polynomial. In particular, $e_0=\overline{e}_0$ and $e_1=\overline{e}_1$. 

\medskip

Consider now the ring $R$ obtained adjoining two indeterminates $U$ and $V$. Thus $R \cong k[\![x,y,z,U,V]\!]$, where $x,y,$ and $z$ denote the 
images of $X, Y$ and $Z$,respectively, modulo the ideal $(Z^2-X^3Y^2,Y^3-XZ, X^4-YZ)$. Let ${\mathfrak n}$ denote the maximal $(x,y,z,U,V)$ 
of $R$ and observe that $J=(x,U,V)$ is a minimal reduction of ${\mathfrak n}$. In addition to the ${\mathfrak n}$-adic and the integral closure filtrations,
${\mathbb N} =\{ {\mathfrak n}^n \}$ and ${\mathbb F}= \{ \overline{{\mathfrak n}^n} \}$ respectively, we also consider the following filtration ${\mathbb G}$ defined by
\[
G_0=R, \quad G_1 = {\mathfrak n}, \quad G_2 = ({\mathfrak n}^2, z), \quad \text{and} \quad G_n = J^{n-2}G_2 
\]
for all $n \geq 3$. Observe that $G_2 \cap J =JG_1$ so that  the associated graded ring of  the filtration 
${\mathbb G}$ is Cohen-Macaulay by the Valabrega-Valla criterion. Moreover, the Rees algebra ${\mathcal R}({\mathbb G})$
of the filtration ${\mathbb G}$ is also Cohen-Macaulay since the reduction 
number of the filtration is $2$ and it is strictly smaller than the dimension of $R$.

We now claim that the filtration ${\mathbb G}$ is actually the normal filtration ${\mathbb F}$. In fact, it is easy to observe that
\[
G_n = {\mathfrak n}^n+z\cdot (U,V)^{n-2}  
\]
for $n \geq 3$. Thus, going modulo $U$ and $V$, we obtain the equalities $G_n \cdot S = {\mathfrak m}^n = \overline{{\mathfrak m}^n}$ in the 
one-dimensional ring $S$. This gives us that 
$e_0=e_0({\mathbb G})=\overline{e}_0$ and  $e_1=e_1({\mathbb G})=\overline{e}_1$. Since we have the inclusion of Rees algebra
${\mathcal R}({\mathbb G}) \subset \overline{\mathcal R}$ with ${\mathcal R}({\mathbb G})$ Cohen-Macaulay and $e_1({\mathbb G})=\overline{e}_1$, by 
\cite[Theorem 2.2]{PUV} we conclude that the filtration ${\mathbb G}$ is the normal filtration. In particular $\overline{\mathcal G}$ is Cohen-Macaulay.

Finally, by \cite[Proposition 4.6]{HM} (see also \cite[Proposition 1.9]{GR} for a simpler proof), we obtain that $\overline{e}_3=0$, as $\overline{\mathcal G}$ 
is Cohen-Macaulay and the reduction number of the normal filtration is $2$. However ${\mathcal G}({\mathfrak m})$ is not Cohen-Macaulay. }
\end{example} 

\noindent{\bf Acknowledgments.} 
This work started at the University of Kentucky in 2011 where the three authors spent a week sponsored by the
NSF DMS-0753127 grant, Special Algebra Meetings in the Midwest. 
The authors are very appreciative of the hospitality offered by this institution. Part of this work was continued at the Mathematical Sciences Research Institute (MSRI) in Berkeley, 
where the authors spent  time in connection with the 2012-13 thematic year on Commutative Algebra, supported by NSF grant 0932078000. 
The authors would like to thank MSRI for its partial support and hospitality.

\end{document}